\documentclass{article}
\usepackage[utf8]{inputenc}
\usepackage{amsmath, amssymb, enumerate, multicol, chngcntr}
\usepackage[exercise]{amsthm}
\usepackage{xcolor}
\usepackage{pdfpages}
\usepackage{comment}
\usepackage{caption}
\usepackage{bm}
\usepackage{mathtools}

\theoremstyle{plain}
\newtheorem{thm}{Theorem}[section]

\newtheorem{lemm}[thm]{Lemma}

\counterwithin{equation}{section}

\usepackage{fancyhdr}
\usepackage{setspace}

\usepackage[english]{babel}
\usepackage[euler]{textgreek}
\usepackage{stmaryrd}

\usepackage{tikz}
\usetikzlibrary{shapes,arrows}
\usepackage{tikz-cd}
\usepackage{amssymb}
\usepackage{graphicx}
\graphicspath{{Images/}}
\usepackage[font=small,labelfont=bf]{caption}
\usepackage{pgf,tikz,pgfplots}
\pgfplotsset{compat=1.15}
\usepackage{mathrsfs}
\usepackage{diagbox}
\usepackage{float}
\usepackage{xcolor}
\usepackage{hyperref}
\usepackage{ragged2e}
\usepackage{enumerate}
\usepackage{upgreek}
\usepackage{multicol}

\theoremstyle{remark}

\theoremstyle{plain}

\newtheoremstyle{note}
  {3pt}
  {3pt}
  {}
  {}
  {\itshape}
  {:}
  {.5em}
  {}

\newtheoremstyle{citing}
  {3pt}
  {3pt}
  {\itshape}
  {}
  {\bfseries}
  {.}
  {.5em}
  {\thmnote{#3}}

\theoremstyle{citing}

\newtheoremstyle{break}
  {9pt}
  {9pt}
  {\itshape}
  {}
  {\bfseries}
  {.}
  {\newline}
  {}

\let\lvert=|\let\rvert=|

\title{Conjugacy classes of completely reducible cube-free solvable $p'$-subgroups of $\mbox{GL}(2, q)$}
\author{Prashun Kumar \footnote{Dr. B. R. Ambedkar University Delhi, Delhi 110006, India; \ E-mails: prashunkumar.19@stu.aud.ac.in,  prashun07kumar@gmail.com.} 
\ and \ Geetha Venkataraman\footnote{Corresponding Author, Dr. B. R. Ambedkar University Delhi, Delhi 110006, India; E-mails: geetha@aud.ac.in,  geevenkat@gmail.com.}}

\begin{document}
\fontfamily{cmr}\selectfont

\maketitle


\bigskip
\noindent
{\small{\bf ABSTRACT:}}
 Let $m$ be a cube-free positive integer and let $p$ be a prime such that $p \nmid m$. In this paper we find the number of conjugacy classes of completely reducible solvable cube-free subgroups in ${\rm GL}(2,q)$ of order $m$, where $q$ is a power of $p$.


\medskip
\noindent
{\small{\bf Mathematics Subject Classification-MSC2020}{:} }
20E34, 20E45, 20F16, 20H30

\vspace{.25in}
\noindent
{\bf THIS IS AN EARLY VERSION OF THE PAPER. FOR THE FINAL VERSION SEE \url{https://doi.org/10.1142/S0219498825502597}.}
\vspace{.25in}

\baselineskip=\normalbaselineskip
\section{Introduction}
A closed formula for the number of conjugacy classes of the reducible subgroups of ${\rm GL}(2,t)$  of orders $p,p^2,pr$ where $p,r$ and $t$ are distinct primes has been given in \cite{DEP2022}. Let $p$ be a prime and let $q$ be a power of $p$. Motivated by the aforementioned result we found a formula for the number of conjugacy classes of reducible cyclic subgroups of ${\rm GL}(2,q)$, see \cite{PKGV2023}.

Chapters $3$ and $4$ of \cite{S1992} give a complete and irredundant list of conjugacy class representatives of soluble irreducible subgroups of ${\rm GL}(2, p^k)$ where $p$ is prime. Subgroups of ${\rm GL}(2,q)$ in general, are also discussed in some detail in \cite{B1967} and \cite{FB2005}.

A group is said to be cube-free if its order is not divisible by the cube of any prime.  The structure of a solvable cube-free $p'$-subgroup of ${\rm GL}(2,q)$ is discussed in \cite{DE2005} and \cite{QL2011}.  The objective of this paper is to use this structure to find the number of conjugacy classes of solvable cube-free $p'$-subgroups of ${\rm GL}(2,q)$ of order $m$ where $p \nmid m$.

\medskip
\noindent
 Throughout the paper, $p$ is a prime, $q$ is a power of $p$ and $\mathbb{F}_q$ is the finite field of order $q$. Let $D(2,q)$, denote the subgroup of diagonal matrices of $\mbox{GL}(2,q)$. Any $d \in D(2,q)$ with diagonal entries $d_1$ and $d_2$ will be represented as $dia(d_1,d_2)$. Let $M(2,q) = D(2,q) \rtimes \langle a \rangle$ be the subgroup of monomial matrices in ${\rm GL}(2,q)$, where $a =$ $\begin{pmatrix}
    0 & 1\\
    1 & 0
\end{pmatrix}$. By $D(2,q)a$ we mean the right coset of $D(2,q)$ with respect to $a$. Let $N(2,q)$ be  the normaliser of $S(2,q)$, where $S(2, q) \cong \mathbb{Z}_{q^2-1}$ is a Singer cycle.

\bigskip
\noindent
Let $H$ be a solvable cube-free $p'$-subgroup of ${\rm GL}(2,q)$. Lemma \ref{strctre_of_cbe_fre_sbgrps_of_GL} below describes the structure of such an $H$. While most of this is known, we nevertheless provide a sketch proof. The main results of this paper will be stated using the structure described in Lemma \ref{strctre_of_cbe_fre_sbgrps_of_GL}.

\begin{lemm}{\label{strctre_of_cbe_fre_sbgrps_of_GL}}
Let $K \leq {\rm GL}(2,q)$ be a solvable cube-free $p'$-subgroup. Then one of the following holds. 
\begin{enumerate}[\rm{(}a\rm{)}]
    \item If $K$ is reducible, then $K$ is conjugated to a subgroup of $D(2,q)$ and $K \cong {\mathbb Z}_l \times {\mathbb Z}_s$ where $l \mid q-1$ and $s \mid q-1$.
    \item If $K$ is imprimitive, then $K$ is conjugated to a subgroup of $M(2,q)$ and $K \cong L \rtimes P $ where $L \leq D(2,q)$ and $P$ is a cyclic subgroup of order $2^{\beta}$ where $\beta \in \{1,2\}$. 
    \item If $K$ is primitive, then $K$ is conjugated to a subgroup of $N(2,q)$ and $K$ is either cyclic or $K = L \rtimes P$ where $L \leq S(2,q)$ and $P$ is a Sylow $2$-subgroup of $K$.
\end{enumerate}
\end{lemm}

\begin{proof}
    If $K$ is a reducible $p'$-subgroup of ${\rm GL}(2,q)$, then the underlying ${\mathbb F}_q K$-module $V$ is a direct sum of two one-dimensional submodules of $K$. So we can find a basis of $V$ with respect to which elements of $K$ are diagonal. Thus $K$ conjugates to a subgroup of $D(2,q)$ and is as given in part (a).
    
    \vspace{.05in}
    Now let $K$ be an imprimitive $p'$-subgroup. Then the underlying ${\mathbb F}_q K$-module $V$ is a direct sum of two one-dimensional subspaces $V_1 = \langle v_1 \rangle$ and $V_2= \langle v_2 \rangle$ such that $K$ permutes the $V_i$. If we choose the basis $\{v_1, v_2\}$ for $V$, then with respect to this basis, the elements of $K$ are either diagonal or are elements of the coset $D(2,q)a$. Hence $K$ conjugates to a subgroup of $M(2,q)$. Now assume $K \leq M(2,q)$. Then $\hat{K} = K \cap D(2,q)$ is a proper normal subgroup of $K$. Let $L_1$ be the Hall $2'$-subgroup of $\hat{K}$. Then $K = L_1 \rtimes P_1$, where $P_1$ is a Sylow $2$-subgroup of $K$ and using this we can write $K$ in the required form.  

\vspace{.05in}
Now let $K$ be a primitive solvable cube-free $p'$-subgroup of ${\rm GL}(2,q)$. If $K$ is abelian, then $K$ is cyclic and by \cite[Theorem 2.3.2]{S1992} and \cite[Theorem 2.3.3]{S1992}, $K$ is conjugated to a subgroup of $N(2,q)$. Suppose $K$ is non-abelian. Let $F = F(K)$ be the Fitting subgroup of $K$. Since $K$ is of cube-free order, $F$ is abelian. By Clifford's Theorem, we get that $F$ is either irreducible or $F$ has only scalar matrices. Since $K$ is solvable we have $C_K(F) \leq F$. Thus $F$ cannot have only scalar matrices and must be irreducible. Since $F$  is abelian, it has to be cyclic. Therefore as seen earlier, $F$ is conjugated to a subgroup of $S(2,q)$. Since $F \trianglelefteq K$, by \cite[Theorem 2.3.5]{S1992}, we have that $K$ is conjugated to a subgroup of $N(2,q)$.  Since $N(2,q) = S(2, q) \rtimes \langle b \rangle$ where $b$ has order $2$, as in the above case, we can show that $K$ has the form as in part (c) if $K$ has an element in common with the coset $S(2,q)b$.

\end{proof}

 Now we shall state the main results of this paper using the results of Lemma \ref{strctre_of_cbe_fre_sbgrps_of_GL}.

\begin{thm}{\label{cnjugcy_clss_of_rducble_sbgrps}}
    Let $H$ be a subgroup of $D(2, q)$ of cube-free order $m$ where $p\nmid m$. Let $m= p_0^{\beta_0}p_1^{\beta_1} \ldots p_k^{\beta_k}$ be the prime decomposition for $m$ where $p_0=2$. Further let $\beta_i$ be integers with  $\beta_i \geq 0$ for all $i$ and at least one $\beta_i > 0$. If $\beta_i >0$, then let $P_i$ denote a Sylow $p_i$-subgroup of $H$. Let ${\cal I} =\{i >0 \mid P_i \mbox{ is cyclic} \}$ and let $|{\cal I}| =r$.

    \vspace{.05in}
    \noindent Let $N_{red}(m,H)$ be the number of conjugacy classes of reducible subgroups of ${\rm GL}(2,q)$ of order $m$ that are isomorphic to $H$.  Then 
  $$N_{red}(m,H) = \frac{1}{2}(\rho(m,H) + \delta(m,H))$$
  where $\rho(m,H) =
  \begin{cases}

  \displaystyle\prod_{i \in {\cal I} \cup\{0\}}(p_i^{\beta_i} + p_i^{\beta_i -1}) & \mbox{if } r \geq 0, \, m \geq 2 \mbox{ is even and } P_0 \mbox{ is cyclic}, \smallskip \\
  \ \ \, \displaystyle\prod_{i \in {\cal I}}(p_i^{\beta_i} + p_i^{\beta_i -1}) & \mbox{if } r > 0, \, m >2 \mbox{ is odd or } P_0 \cong {\mathbb Z}_2 \times {\mathbb Z}_2, \smallskip \\
  \ \ \ 1 & \mbox{if } r = 0, \, \beta_0=0  \mbox{ or } P_0 \cong {\mathbb Z}_2 \times {\mathbb Z}_2,
  \end{cases}$
  
  \vspace{.25in}
  \noindent
  \ \ and 
    $\delta(m,H) =
\begin{cases}
 2^r  & \mbox{if } r \geq 0, \, 0 \leq \beta_0 \leq 1 \mbox{ or } P_0 \cong {\mathbb Z}_2 \times {\mathbb Z}_2, \smallskip \\
  2^{r+1} & \mbox{if } r \geq 0, \, \beta_0 = 2 \mbox{ and } P_0 \cong {\mathbb Z}_4.
\end{cases}
$\\

\end{thm} 

\begin{thm}\label{Conjgcy_clases_of_imprmtve_sbgrps}
    Let $H\leq M(2,q)$ be a cube-free imprimitive subgroup of order $m$ where $p\nmid m$. Let $N_{imp}(m, H)$ be the number of  conjugacy classes of imprimitive subgroups of ${\rm GL}(2,q)$ of order $m$ that are isomorphic to $H$. Then $N_{imp}(m,H)=1$.
\end{thm}

\begin{thm}{\label{Conjgcy_clases_of_prmtve_sbgrps}}
    Let $H\leq N(2,q)$ be a cube-free primitive subgroup of order $m$ where $p\nmid m$. Let $N_{pr}(m, H)$ be the number of conjugacy classes of imprimitive subgroups of ${\rm GL(2,q)}$ of order $m$ that are isomorphic to $H$. Then $N_{pr}(m, H) = 1$.
\end{thm}

The paper is organised as follows. We prove Theorem \ref{cnjugcy_clss_of_rducble_sbgrps} in Section \ref{Red_sbgrps_of_GL(2,p)}. In Section \ref{imprmtve_sbgrps} we find the conjugacy classes in $M(2, q)$ of elements of orders $2$ and $4$ and then prove Theorem \ref{Conjgcy_clases_of_imprmtve_sbgrps}. In Section \ref{prmtve_sbgrps} we find the number of conjugacy classes in $N(2,q)$ of elements of orders $2$ and $4$ and prove Theorem \ref{Conjgcy_clases_of_prmtve_sbgrps}. Finally, in Section \ref{misc} we provide an explicit description of the cube-free solvable $p'$-subgroups of ${\rm GL}(2,q)$ which can be taken as representatives of the conjugacy classes.

\section{Reducible cube-free $p'$-subgroups of ${\mbox GL}(2,q)$}\label{Red_sbgrps_of_GL(2,p)}
In this section we will provide a closed formula for the number of conjugacy classes of reducible cube-free $p'$-subgroups of ${\rm GL}(2,q)$. Let $K$ be a reducible subgroup of ${\rm GL}(2,q)$ of order $m$ where $p \nmid m$ and $m$ is cube-free. By Lemma \ref{strctre_of_cbe_fre_sbgrps_of_GL}, we know that $K$ will be conjugate to a subgroup of $D(2, q)$. 

\medskip
\noindent
{\bf Proof of Theorem \ref{cnjugcy_clss_of_rducble_sbgrps}} 

\begin{proof}Fix the subgroup $H$ of $D(2,q)$ of order $m$ where $p\nmid m$ and where $m$ is cube-free. Let ${\cal Y} = \{K \leq {\rm GL}(2,q) \mid K \mbox{ is reducible and } K \cong H \}$. Then ${\rm GL}(2,q)$ acts on ${\cal Y}$ by conjugation. Let ${\cal \hat{Y}} = \{ [K] \mid  K \in {\cal Y}\}$. Clearly $N_{red}(m, H) = |{\cal \hat{Y}}|$.

Let ${\cal Y}_M = \{ T \mid T \leq D(2,q) \mbox{ and } T \cong H\}$. Then $M(2, q)$ acts on ${\cal Y}_M$ by conjugation. Let ${\cal \hat{Y}}_M = \{ [T]_M \mid T \leq D(2,q) \mbox{ and } T \cong H\}$ where $[T]_M$ denotes the conjugacy class of $T$ with respect to the action of $M(2, q)$. 

    \vspace{.05in}
    \noindent
    We know that any reducible subgroup of ${\rm GL}(2,q)$ whose order is co-prime to $p$ is conjugate to a subgroup of $D(2,q)$. So for $K \leq {\rm GL}(2,q)$ such that $[K] \in {\cal \hat{Y}}$ there exists a $\hat{K} \leq D(2,q)$ such that ${\hat{K} \in [K]}$.
    Further two distinct subgroups of $D(2,q)$ that are conjugates in ${\rm GL}(2,q)$ are always conjugated in $M(2,q)$, see \cite[Lemma 1.3]{PKGV2023}. Thus the map from ${\cal \hat{Y}}$ to ${\cal \hat{Y}}_M$ given by $[K] \to [\hat{K}]_M$ turns out to be bijective. Hence we can conclude that $N_{red}(m,H) = |{\cal \hat{Y}}| = |{\cal \hat{Y}}_M|$.
    
    \vspace{.05in}
    \noindent
     Any abelian group is a direct product of its Sylow subgroups. Thus $|{\cal Y}_M|= \prod_{i=0}^{k} t_i$, where $t_i$ is the number of subgroups of order $p_i^{\beta_i}$ in $D(2,q)$. Since $H$ is a cube-free group, the Sylow $p_i$-subgroup of $H$ is either cyclic or isomorphic to ${\mathbb Z}_{p_i} \times {\mathbb Z}_{p_i}$. Further by Lemma \ref{strctre_of_cbe_fre_sbgrps_of_GL}, we have that $H \cong {\mathbb Z}_l \times {\mathbb Z}_s$ where $l \mid q-1$ and $s \mid q-1 $. So $p_i \mid q-1$ for all $i$.

    \vspace{.05in}
    \noindent
     If $P_i \cong {\mathbb Z}_{p_i} \times {\mathbb Z}_{p_i}$, then there is only one choice for $P_i$ as a subgroup of $D(2,q)$, see \cite[Lemma 1.2]{PKGV2023}. Therefore $|{\cal Y}_M|= \prod_{i \in {\cal I} \cup \{0\}} t_i$. The product will not involve $t_0$ if either $\beta_0 = 0$ or $P_0 \cong {\mathbb Z}_2 \times {\mathbb Z}_2$.

    \vspace{.05in}
    \noindent
    Now a cyclic subgroup of order $p_i^{\beta_i}$ in $D(2, q)$ is generated by an element of the form $dia(\lambda_1,\lambda_2)$ where $\lambda_i \in {\mathbb F}_q^*$. Further the order of one of the $\lambda_i$ is $p_i^{\beta_i}$ and the order of the other divides $p_i^{\beta_i}$. Therefore
    
    \begin{align*}
t_i & = \frac{{(\varphi(p_i^{\beta_i}))}^2 + 2\sum_{j=1}^{\beta_i} \varphi(p_i^{\beta_i})\varphi(p_i^{\beta_i-j})}{\varphi(p_i^{\beta_i})}\\
    & = \varphi(p_i^{\beta_i}) + 2 \{\varphi(p_i^{\beta_i -1}) + \ldots + \varphi(p_i) + 1\}\\
    & = p_i^{\beta_i} + p_i^{\beta_i - 1}
\end{align*}
where $\varphi$ is the Euler's $\varphi$-function. Hence $|{\cal Y}_M|= \prod_{i \in {\cal I} \cup \{0\}}(p_i^{\beta_i} + p_i^{\beta_i - 1})$ provided $\beta_0 \geq 1$ and $P_0$ is cyclic. If not, the product will only involve $i \in {\cal I}$. By \cite[Theorem 3.22]{JR1995}, the number of orbits required
\begin{equation}{N_{red}(m,H)} \label{eq: no_of_orbts}
 = \frac{1}{2|D(2,q)|} \, \left (\sum_{d\in D(2,q)}  | {\rm Fix}(d)| + \sum_{d\in D(2,q)} |{\rm Fix}(da)|\right). \tag{$*$}
\end{equation}
Clearly each $d \in D(2,q)$ fixes every element of ${\cal Y}_M$. So $|{\rm Fix}(d)|= |{\cal Y}_M|$. Also ${\rm Fix}(da) = {\rm Fix}(a)=\{K \in {\cal Y}_M \mid aK{a}^{-1} = K \}$. Now let $S_i = \{S \leq D(2,q) \mid S \cong P_i \mbox{ and } aS{a}^{-1} = S\}$. Therefore $|{\rm Fix}(a)|= \prod_{i=0}^{k} |S_i|$ where $i$ occurs in the product only if $\beta_i >0$.

\vspace{.05in}
\noindent
As seen earlier if $P_i \cong {\mathbb Z}_{p_i} \times {\mathbb Z}_{p_i}$ for any $i$, then $|S_i|=1$. So $|{\rm Fix}(a)| = \prod_{i \in {\cal I} \cup \{0\}} |{S_i}|$ provided $\beta_0 \geq 1$ and $P_0$ is cyclic. If not, the product will only involve $i \in {\cal I}$.

\vspace{.05in}
\noindent
Now for any $i$ if $P_i$ is cyclic, then by \cite[Lemma 2.2]{PKGV2023}, we get that $|S_i| = 1 +  \mbox{Number of elements of order } 2 \mbox{ in } {\rm Aut}(\mathbb{Z}_{p_i^{\beta_i}})$. Thus $|S_i| = 2$ for $i \in {\cal I}$. Further if $0 \leq \beta_0 \leq 1$ or $P_0 \cong {\mathbb Z}_2 \times {\mathbb Z}_2$ then $|S_0| = 1$ and $|S_0| = 2$ if $P_0 \cong {\mathbb Z}_4$. Putting these values in (\ref{eq: no_of_orbts}) we get the desired value of ${N_{red}(m,H)}$. 
\end{proof}

\section{Imprimitive cube-free $p'$-subgroups of ${\rm GL}(2,q)$ }
\label{imprmtve_sbgrps}

In this section we will determine the number of conjugacy classes of cube-free solvable imprimitive $p'$-subgroups of ${\rm GL}(2,q)$. Let $K$ be a solvable imprimitive subgroup of ${\rm GL}(2,q)$ of cube-free order $m$ where $p \nmid m$. Then by Lemma \ref{strctre_of_cbe_fre_sbgrps_of_GL}, $K$ is a conjugate of a subgroup $H$ of $M(2, q)$. Further, $H = L \rtimes P$ where $L \leq D(2,q)$ and $P$ is a cyclic subgroup of order $2^{\beta}$ of $H$ where $\beta \in  \{1,2\}$. We will use this structure to show that any two isomorphic cube-free solvable imprimitive $p'$-subgroups of ${\rm GL}(2,q)$ are conjugate in ${\rm GL}(2,q)$. 

\vspace{.05in}

\begin{lemm}\label{elmnts_of_ordr_4_in_M}
    Let $g$ and $h$ be any two elements of order $t$ in the coset $D(2,q)a$, where $t \in \{2, 4\}$. Then there exists an element $d \in D(2,q)$ such that $dgd^{-1} = h$. Thus the elements of order $t$ in $D(2,q)a$ form a single conjugacy class in $M(2, q)$. 
\end{lemm}
\begin{proof}
   Let $g \in M(2,q)$ belong to the coset $D(2,q)a$. If $g$ has order $2$, then $g = dia(\lambda,\lambda^{-1})a$ and if $g$ has order $4$, then $g = dia(\lambda,u\lambda^{-1})a $ where $\lambda \in {\mathbb F}_q^*$ and $u \in {\mathbb F}_q^*$ is the unique element of order $2$. (Note that an element of order $4$ exists in $D(2, q)a$ only if $q$ is odd.)

   \vspace{.05in}
\noindent
   Let $d = dia(\lambda,1)$ and let $g = dia(\lambda, \lambda^{-1})a$ be of order $2$. Then $dad^{-1}=g$. Similarly let $k = dia(1,u)a$ and let $h = dia(\lambda,u\lambda^{-1})a$ be of order $4$.  Then $dkd^{-1} = h$. 
\end{proof}

\begin{lemm}{\label{sbgrps_of_M(2,q)}}
    Let $H_1 = L_1 \rtimes P_1 $ and $H_2 = L_2 \rtimes P_2$ be two imprimitive subgroups of $M(2,q)$, where $L_1$ and $L_2$ are subgroups of $D(2,q)$ and $P_1$ and $P_2$ are cyclic subgroups  order $2^{\beta}$ where $\beta \in  \{1,2\}$. Then $H_1 \cong H_2$ if and only if $L_1 = L_2$ and $P_1 \cong P_2$.
\end{lemm}

\begin{proof}
    Let $\phi: H_1 \longmapsto H_2$ be an isomorphism. We first claim that $\phi(L_1) = L_2$.

    If the $H_i$ are non-abelian then the $L_i$ are either the Hall $2'$-subgroups respectively or they are the respective Fitting subgroups and so $\phi(L_1) = L_2$. If the $H_i$ are abelian, then either the $L_i$ are Hall $2'$-subgroups respectively and so $\phi(L_1) = L_2$ or we can write $\phi(L_1) = L \times \langle d_1 \rangle$ and $L_2 = L \times \langle d_2 \rangle$ where $L$ is the Hall $2'$-subgroup of $H_2$ and  $d_1$ and $d_2$ are elements of order $2$ in $D(2,q)$. Now if $P_2 = \langle d'a \rangle$, for some $d' \in D(2, q)$, then we have that $d_i$ commute with $d'a$. Thus we get $ad_ia^{-1}=d_i$ and so $d_i$ is a scalar matrix of order $2$ for each $i$. Hence $d_1 =d_2$ and we get $\phi(L_1) = L_2$ as required. This also implies that $|P_1| = |P_2|$ and so $P_1 \cong P_2$.

    Now let $s$ be a prime divisor of $|L_1|$ and let $S$ be the Sylow $s$-subgroup of $L_1$ with $|S| = s^k$ for some $k \in \{1, 2\}$. Our aim is to show that $S = \phi(S)$ for each prime $s$ dividing $|L_1|$ giving us $L_1 = \phi(L_1) = L_2$. 

If $S \cong {\mathbb Z_s} \times {\mathbb Z_s}$ by \cite[Lemma 1.2]{PKGV2023} we have that $S$ is the unique subgroup of $D(2, q)$ that is isomorphic to ${\mathbb Z_s}\times {\mathbb Z_s}$. Thus we must have $\phi(S) = S$.

Now let $S$ be cyclic. Let $P_1 = \langle da \rangle$ for some $d \in D(2,q)$. Since $L_1 \trianglelefteq H_1$, we have $aL_1{a}^{-1} = L_1 $ and so $aS{a}^{-1} = S$. Since $aS{a}^{-1} = S$, by \cite[Lemma 2.2]{PKGV2023} we get that $S = \langle dia(\lambda_1, {\lambda_1}^{l_1}) \rangle$ where $\lambda_1 \in {\mathbb{F}_q}^{*}$ with $o(\lambda_1)= s^k$ and $l_1 \in \rm{Aut}(\mathbb{Z}_{s^k})$ with ${l_1}^2= 1$.  Similarly we must have $\phi(S) = \langle dia(\lambda_2, {\lambda_2}^{l_2}) \rangle$ where $\lambda_2 \in {\mathbb{F}_q}^{*}$ with $o(\lambda_2)= s^k$ and $l_2 \in \rm{Aut}(\mathbb{Z}_{s^k})$ with ${l_2}^2= 1$. 

If $S \not = \phi(S)$ then by \cite[Lemma 2.2]{PKGV2023}, we must have $l_1 \not = l_2$. Since ${l_i}^2 =1$ in $\rm{Aut}(\mathbb{Z}_{s^k})$, we can assume that $l_1 =1$ and that $l_2 = -1$. But then $S$ is generated by a scalar matrix and is central. Using this we can show that a generator for $\phi(S)$ is a scalar matrix which is a contradiction. Hence we must have $S = \phi(S)$ when $S$ is cyclic.   

If $L_1 = L_2$ and $\psi : P_1 \longmapsto P_2$ is an isomorphism, then we can define a map $f : H_1 \longmapsto H_2$ as $f(bz) = b\psi(z)$ where $b \in L_1$ and $z \in P_1$. One can easily check that $f$ is an isomorphism as $yxy^{-1} = \psi(y) x {\psi(y)}^{-1}$ for all $x \in L$ and $y \in P_1$.
\end{proof}

\noindent
{\bf Proof of Theorem \ref{Conjgcy_clases_of_imprmtve_sbgrps}}

\begin{proof}
Let $H\leq M(2,q)$ be a cube-free imprimitive subgroup of order $m$ where $p\nmid m$. By Lemma \ref{strctre_of_cbe_fre_sbgrps_of_GL}, we can assume that $H = L \rtimes P$ where $L \leq D(2,q)$ and $P$ is cyclic of order $2^k$ where $k \in \{1, 2\}$.

Let $H_1$ be an imprimitive subgroup of $M(2,q)$ isomorphic to $H$. Then by Lemma \ref{sbgrps_of_M(2,q)}, we get $H_1 = L \rtimes P_1$ where $P_1 \cong P$. Further by Lemma \ref{elmnts_of_ordr_4_in_M}, we must have $P_1 = dPd^{-1}$ for some $d \in D(2,q)$ and hence we have $dHd^{-1} = H_1$. So every imprimitive subgroup of $M(2,q)$ which is isomorphic to $H$ is conjugate to $H$. 

Now suppose $K$ is an imprimitive subgroup of ${\rm GL}(2,q)$ isomorphic to $H$. Then by Lemma \ref{strctre_of_cbe_fre_sbgrps_of_GL} there exist a subgroup $H_1$ of $M(2,q)$ such that $K$ is conjugate to $H_1$ in ${\rm GL}(2,q)$. Clearly by the above discussion $H_1$ is a conjugate of $H$. Thus every imprimitive subgroup of ${\rm GL}(2,q)$ isomorphic to $H$ is also a conjugate of $H$. 
\end{proof}

\section{Primitive cube-free $p'$-subgroups of ${\rm GL}(2,q)$ that are solvable}\label{prmtve_sbgrps}

Let $H\leq N(2,q)$ be a cube-free primitive subgroup of order $m$ where $p\nmid m$. In this section we use the structure of cube-free solvable primitive subgroups of ${\rm GL}(2,q)$ to show that $N_{pr}(m, H) = 1$. Recall that $N(2,q) = S(2, q) \rtimes \langle b \rangle$ where $b$ has order $2$. Further using the discussion after Theorem 2.3.5 in \cite{S1992}, we have that $b h b^{-1} = h^q$ where $S(2,q) = \langle h \rangle$. Also note that unless $p$ is an odd prime $N(2, q)$ cannot have elements of order $4$.

\begin{lemm}{\label{elmts_of_ordr_2_and_4_in_N}} Any two elements of order $s$ in the coset $S(2,q)b$, where $s \in \{2,4\}$, are conjugate in $N(2,q)$. Thus the elements of order $s$ in $S(2,q)b$ form a single conjugacy class.
\end{lemm}

\begin{proof}
    The action of $b$ on $h$ ensures that only elements of the form $h^{i(q-1)}b$ where $0 \leq i < q+1$ have order $2$ in $S(2,q)b$ and that every such element is conjugate to $b$.  
    
    Let $p$ be an odd prime. Then we can show that an element of order $4$ in $S(2,q)b$ will have the form $h^{l(q-1)/2}b$ where $l$ is odd and $1 \leq l <2(q+1)$. Let $g = h^{(q-1)/2}b$. Then we can see that $C_{N(2,q)}(g) = \langle g \rangle \langle h^{q+1} \rangle$ and hence $|[g]| = q+1$. Since $l$ is odd, we get $[g]$ is precisely the set of elements of order $4$ in $S(2,q)b$. 
\end{proof}

\medskip
\noindent
{\bf Proof of Theorem \ref{Conjgcy_clases_of_prmtve_sbgrps}}

\begin{proof}
Let $H\leq N(2,q)$ be a cube-free primitive subgroup of order $m$ where $p\nmid m$. By Lemma \ref{strctre_of_cbe_fre_sbgrps_of_GL}, we can assume that if $H \leq S(2,q)$ then $H$ is cyclic. Otherwise we can write $H = L \rtimes P$ where $L \leq S(2,q)$ and $P$ is a Sylow $2$-subgroup of $H$.

Now let $H_1$ be a cube-free primitive subgroup of $N(2,q)$ which is isomorphic to $H$. If $H$ is a subgroup of $S(2,q)$ then it is a cyclic irreducible subgroup of order $m$. By \cite[Theorem 2.3.3]{S1992}, all cyclic irreducible subgroups of order $m$ form a single conjugacy class in ${\rm GL}(2, q)$ and so we have that $H$ and $H_1$ are conjugate. 

Let us assume now that $H$ is not cyclic, and that $H = L \rtimes P$ as above. Since $H_1 \cong H$, we must have that $H_1 = L_1 \rtimes P_1$ where $P_1$ is a Sylow $2$-subgroup of $H_1$ and $L_1 \leq S(2,q)$. 

If $P$ is elementary abelian of order $4$, then $|P \cap S(2,q)| =2$. So we can write $H = (L \times P \cap S(2,q)) \rtimes \langle u \rangle$ where $u \in S(2,q)b$ is of order $2$. Similarly $H_1 = (L_1 \times P_1 \cap S(2,q)) \rtimes \langle v \rangle$ where $v \in S(2,q)b$ is of order $2$. Thus by Lemma \ref{elmts_of_ordr_2_and_4_in_N}, we get that there exists $g \in N(2,q)$ such that $gPg^{-1} = P_1$ whether $P$ is cyclic or elementary abelian.  Since $S(2,q)$ is cyclic we have $L = L_1$ and so $gH_1g^{-1} = H_2$. Thus every primitive subgroup of $N(2,q)$ which is isomorphic to $H$ is conjugate to $H$. 
 
 Now suppose $K$ is a primitive subgroup of ${\rm GL}(2,q)$ isomorphic to $H$. Then by Lemma \ref{strctre_of_cbe_fre_sbgrps_of_GL} there exist a subgroup $H_1$ of $N(2,q)$ such that $K$ is conjugate to $H_1$ in ${\rm GL}(2,q)$. Clearly by the above discussion $H_1$ is a conjugate of $H$. Thus every primitive subgroup of ${\rm GL}(2,q)$ isomorphic to $H$ is also a conjugate of $H$. 
\end{proof}

\section{Miscellaneous}
\label{misc}
In this section we provide an explicit description of the cube-free solvable $p'$-subgroups of ${\rm GL}(2,q)$ which can be taken as representatives of the conjugacy classes. By Lemma \ref{strctre_of_cbe_fre_sbgrps_of_GL}, we can consider these as members of $D(2,q)$, $M(2,q)$ and $N(2, q)$ respectively when they are reducible, imprimitive and primitive respectively.

We first consider $H$ as given in Theorem \ref{cnjugcy_clss_of_rducble_sbgrps}. The notations established there will be used as will some aspects of the proof. Let $H = \prod_{i=0}^{k}P_i$ where $P_i$ is the Sylow $p_i$-subgroup of $H$ and the product is direct. Let $M = M(2,q)$. Then $N_M(H)$ is either $D(2,q)$ or $M(2,q)$. 

Let $N_M(H)= M(2,q)$. Since $aHa^{-1} = H$, we get that $aP_ia^{-1} = P_i$ for all $i$. Now for any $ i \in {\cal I}$ we have that $P_i$ is a cyclic subgroup of $D(2,q)$, satisfying $aP_ia^{-1} = P_i$. Thus by \cite[Lemma 2.2]{PKGV2023}, we get that $P_i = \langle dia(\lambda_i, \lambda_i^{k_i}) \rangle$ with ${k_i}^2=1\mod{{p_i}^{\beta_i}}$ where $|\lambda_i| = {p_i}^{\beta_i}$ and $1 \leq k_i \leq {p_i}^{\beta_i} -1$. So $k_i =1$ or $k_i = {p_i}^{\beta_i} -1$. If $P_0$ is cyclic it will have a similar form.  

Let ${\cal I} = {\cal I}_1 \cup {\cal I}_2$ where ${\cal I}_1= \{ i \in {\cal I}  \mid k_i =1\}$ and ${\cal I}_2 = {\cal I}\setminus {\cal I}_1$. If $P_0$ is not cyclic, then for all $i \not \in I$, we must have that $P_i$ is the unique subgroup of $D(2,q)$ isomorphic to ${\mathbb Z}_{p_i} \times {\mathbb Z}_{p_i}$. For such $i >0$, we can take $P_i = \langle dia(\lambda_i, \lambda_i) \rangle \times \langle dia(\lambda_i, {\lambda_i}^{p_i-1}) \rangle$ where $|\lambda_i| = p_i$. Let $$H_1 = \prod_{i=1}^{k} P_i = \left(\prod_{\{i \in {\cal I}_1\}} P_i \right) \times \left(\prod_{\{i \in {\cal I}_2 \}} P_i\right) \times \left(\prod_{\{i \in {\cal I}_3\}}P_i\right)$$ where ${\cal I}_3$ consists of $i \not \in {\cal I}$ and $i \neq 0$. For $t \in \{1, 2, 3\}$, define $\lambda_{{\cal I}_t} = \prod_{\{i \in {\cal I}_t\}}\lambda_i$. Note that $\lambda_{{\cal I}_t}$ has order $\prod_{\{i \in {\cal I}_t\}}{p_i}^{\beta_i}$ for $t = 1,2$ and $\lambda_{{\cal I}_3}$ has order $\prod_{\{i \in {\cal I}_3\}}{p_i}$. Clearly $\prod_{\{i \in {\cal I}_1\}} P_i = \langle dia(\lambda_{{\cal I}_1}, \lambda_{{\cal I}_1}) \rangle$. For $i \in {\cal I}_2$, we know that $k_i \not = 1$. Therefore we get that $k_i = {p_i}^{\beta_i} -1$. Since $\lambda_{{\cal I}_2} \in {\mathbb F}_q^*$, it can be shown easily that $\prod_{\{i \in {\cal I}_2\}}{\lambda_i}^{k_i} = {\lambda_{{\cal I}_2}}^{-1}$. Thus $\prod_{\{i \in {\cal I}_2\}} P_i = \langle dia(\lambda_{{\cal I}_2}, {\lambda_{{\cal I}_2}^{-1}}) \rangle$. 

Similarly we can show that $\prod_{\{i \in {\cal I}_3\}} P_i = \langle dia(\lambda_{{\cal I}_3}, \lambda_{{\cal I}_3}) \rangle \times \langle dia(\lambda_{{\cal I}_3}, {\lambda_{{\cal I}_3}}^{-1}) \rangle $. Let $\lambda_{ij} =\lambda_{{\cal I}_i}\lambda_{{\cal I}_j}$ where $i \not= j$ and $i, j \in \{1,2,3\}$. Using the fact that the orders of the $\lambda_{{\cal I}_t}$ are pairwise coprime, we get
$$H_1 = \langle dia(\lambda_{13}, \lambda_{13}) \rangle \times \langle dia(\lambda_{23}, {\lambda_{23}}^{-1}) \rangle.$$ Note that $|\lambda_{t3}| = (\prod_{\{i \in {\cal I}_t\}}{p_i}^{\beta_i})(\prod_{\{i \in {\cal I}_3\}}{p_i})$ for $t \in {1,2}$. Now $H = P_0 \times H_1$ where $H_1$ is as above. If $P_0$ is cyclic then $P_0 = \langle dia(\lambda_0, \lambda_0^{k_0}) \rangle$ with ${k_0}^2=1\mod{2^{\beta_0}}$ where $|\lambda_0| = 2^{\beta_0}$ and $1 \leq k_0 \leq {2}^{\beta_0} -1$. If $P_0$ is not cyclic then $P_0 = \langle dia(-1,-1) \rangle \times \langle dia(-1,1) \rangle$.

\medskip
If $N_M(H) = D(2,q)$, then again $P_i$, the Sylow $p_i$-subgroup of $H$ is either cyclic or a unique subgroup of $D(2,q)$ isomorphic to ${\mathbb Z}_{p_i} \times {\mathbb Z}_{p_i}$. For such $i>1$, we get $P_i = \langle dia(\lambda_i, \lambda_i) \rangle \times \langle dia(\lambda_i, {\lambda_i}^{p_i-1}  ) \rangle$ where $|\lambda_i| = p_i$. Let ${\cal I}_1 = \{i >1 \mid P_i\mbox{ is cyclic and central}\}$. Let ${\cal I}_2 = \{i>1 \mid P_i\mbox{ is cyclic and non-central}\}$ and ${\cal I}_3 = \{i>1 \mid P_i \cong {\mathbb Z}_{p_i} \times {\mathbb Z}_{p_i}\}$. 

\smallskip
For each $i \in {\cal I}_2$ we can show that $P_i = \langle dia(\lambda_i, {\lambda_i}^{k_i}) \rangle$ where $|\lambda_i| = {p_i}^{\beta_i}$, the integer $k_i \in [2, {p_i}^{\beta_i}]$. Let $\lambda' = \prod_{\{i \in {\cal I}_2\}} \lambda_i$. Then it can be shown easily that
$$H = P_0 \times \langle dia(\lambda \lambda'', \lambda \lambda'') \rangle \times \langle dia(\lambda', \prod_{\{i \in {\cal I}_2\}}{\lambda_i}^{k_i}) \rangle \times \langle dia(\lambda'', {\lambda''}^{-1}) \rangle$$ where $\lambda, \lambda''$ are elements of ${\mathbb F}_q^*$ such that $|\lambda| = \prod_{\{i \in {\cal I}_1\}}{p_i}^{\beta_i}$ and $|\lambda''|=\prod_{\{i \in {\cal I}_3\}}p_i$. Note that $P_0$ is either cyclic and central, or cyclic and non-central or elementary abelian of order $4$ and will have an appropriate form as discussed above and in the earlier case.

\medskip
\noindent
Let $H$ be an imprimitive subgroup of $M(2,q)$ of cube-free order $m$ where $p \nmid m$. Let $m= p_0^{\beta_0}p_1^{\beta_1}\ldots p_k^{\beta_k}$  be the prime decomposition of $m$ where $p_0 =2$ and $0 \leq  \beta_i \leq 2 $. Then by Lemma \ref{strctre_of_cbe_fre_sbgrps_of_GL}, we can write $H = L \rtimes P$ where $L \leq D(2,q)$ and $P$ is a cyclic subgroup of order $2^{\beta}$ where $\beta \in \{1,2\}$. Using a proof similar to that of Lemma \ref{sbgrps_of_M(2,q)} we can show that $aLa^{-1} = L$. Thus $L$ is a reducible subgroup of  $D(2, q)$ of cube-free order  with $N_M(L) = M$. Let $P_i$ denote the Sylow $p_i$-subgroups of $L$ for $0\leq i \leq k$. Let ${\cal I}_1 = \{i \geq 1 \mid P_i\mbox{ is cyclic and central}\}$. Let ${\cal I}_2 = \{i \geq 1 \mid P_i\mbox{ is cyclic and non-central}\}$ and ${\cal I}_3 = \{i\geq 1\mid P_i \cong {\mathbb Z}_{p_i} \times {\mathbb Z}_{p_i}\}$. Note that if $|L|$ is even then $P_0$ has to be cyclic of order $2$ and central. By the earlier part, we get that
$$L = P_0 \times \langle dia(\lambda_{13}, \lambda_{13}) \rangle \times \langle dia(\lambda_{23}, {\lambda_{23}}^{-1}) \rangle$$ where $|\lambda_{t3}| = (\prod_{\{i \in {\cal I}_t\}}{p_i}^{\beta_i})(\prod_{\{i \in {\cal I}_3\}}{p_i})$ for $t \in \{1,2\}$. Also note that for these choices of generators for $L$ we do have $aLa^{-1} = L$ since  
$$a{\langle dia({\lambda_{23}}^{-1},\lambda_{23}) \rangle)}a^{-1} 
        =  \langle {dia({\lambda_{23}},\lambda_{23}^{-1})}^{-1} \rangle. $$
Now $H = LP$ and we know that $P$ is cyclic of order $2^{\beta}$ where $\beta \in \{2, 4\}$. Lemma \ref{elmnts_of_ordr_4_in_M} tells us that either $P$ is of order $2$ generated by $ dia(\mu, {\mu}^{-1})a$ or $P$ is of order $4$ generated by $dia(\mu, u{\mu}^{-1})a$ where $u \in {\mathbb F}_q^*$ is the unique element of order $2$. Thus $H$ is determined as a subgroup of $M(2,q)$.

 \vspace{2mm}
 \noindent
 Let $H$ be a cube-free primitive $p'$-subgroup of $N(2,q)$. Then by Lemma \ref{strctre_of_cbe_fre_sbgrps_of_GL}, either $H \leq S(2,q)$ or $H = L \rtimes P$ where $L \leq S(2,q)$ and $P$ is a Sylow $2$-subgroup of $H$ not contained in $S(2,q)$. Note that even when $H \leq S(2,q)$ we can write $H = L \rtimes P = L \times P$ where $P$ is the Sylow $2$-subgroup of $H$. 
 
 If $|L| \mid q-1$ then $L \leq \langle h^{q+1} \rangle$ where $S(2,q) = \langle h \rangle$. Now $\langle h^{q+1} \rangle$ is reducible and conjugates to a subgroup $\hat{K}$ of $D(2,q)$. It is not difficult to show that $\hat{K}$ has no non-scalar matrix. Thus $\langle h^{q+1} \rangle$ is central and so is $L$. 
 
 We can show easily that if $|L| \mid q-1$ then $H$ is not primitive by examining the possibilities for $P$. For, if $P \cong {\mathbb Z}_2 \times {\mathbb Z}_2$ then $H$ is reducible. If $P$ is cyclic and $|P| \mid q-1$ then also $H$ turns out to be reducible. Finally, if $P$ is cyclic and $|P| \nmid q-1 $ then $H$ is irreducible but imprimitive. Thus if $H = L \rtimes P$ is imprimitive then $|L| \mid q^2-1$ but $|L| \nmid q-1$. 

 Conversely, let $H = L \rtimes P$ be a cube-free $p'$-subgroup of order $m$, where $L \leq S(2,q)$ and $P$ is a Sylow $2$-subgroup of $H$. If $|L| \nmid q-1$ then it is not difficult to show that $H$ is primitive.

 \vspace{2mm}
 \noindent
 Now let $m$ be a positive integer such that $m \mid q^2-1$ but $m \nmid q-1$ and let $k = (q^2-1)/m$. Let $S(2,q) = \langle h \rangle$. If $H \leq S(2,q)$ and $|H| = m$, then $H = \langle h^{k} \rangle$. If $H$ is not contained in $S(2,q)$, then  from Lemma \ref{strctre_of_cbe_fre_sbgrps_of_GL} we can take $H = \langle h^{k}\rangle P$ where $P$ is the Sylow $2$-subgroup of $H$. Also, using Lemma \ref{elmts_of_ordr_2_and_4_in_N} we can write down the possible generators of $P$.

\section{Acknowledgements}

\noindent Prashun Kumar would like to acknowledge the UGC-JRF grant ({\emph{identification number}}: 201610088501) which is enabling his doctoral work.


\begin{thebibliography}{50}

\bibitem{B1967} David M Bloom, ‘The Subgroups of $\rm{PSL}(3,q)$ for odd $q$', \emph{Transactions of the American Mathematical Society} 127 (1967) 150--178.
\bibitem{DE2005} Heiko Dietrich and Bettina Eick, ‘On the groups of cube-free order’, \emph{Journal of Algebra} 292 (2005) 122--137.
\bibitem{DEP2022} H. Dietrich, B. Eick and X. Pan, ‘Groups whose orders factorise into at most four primes’, \emph{Journal of Symbolic computation} 108 (2022) 23--40.
\bibitem{FB2005} D. L. Flannery and E. A. O' Brien, ‘The linear groups of small degree over finite fields’, \emph{International Journal of Algebra and Computation} 15 (03) (2005) 467--502.
\bibitem{PKGV2023} P. Kumar and G. Venkataraman,  ‘Conjugacy classes of completely reducible cyclic subgroups of GL(2, q)', \emph{Communications in Algebra}  51 (8) (2023) 3182--3187, DOI: 10.1080/00927872.2023.2179634.
\bibitem{QL2011} S. Qiao and C. H. Li, ‘The finite groups of cube-free order’, \emph{Journal of Algebra} 334 (2011) 101--108.
\bibitem{JR1995} Joseph J. Rotman, {\emph{A course in Abstract Algebra}} (Fourth Edition), Springer-Verlag, New York 1995.
\bibitem{S1992} M. W. Short, \emph{The Primitive Soluble Permutation Groups of Degree less than 256}, Springer-Verlag Heidelberg 1992.

\end{thebibliography}
\end{document}